\documentclass[11pt]{amsart}
\usepackage[dvips]{graphicx,color}
\usepackage{amssymb,amscd,latexsym,epsfig}
\usepackage[all]{xypic}
\usepackage[mathscr]{euscript}
\DeclareFontFamily{OT1}{rsfs}{}
\DeclareFontShape{OT1}{rsfs}{n}{it}{<-> rsfs10}{}
\DeclareMathAlphabet{\curly}{OT1}{rsfs}{n}{it}

\renewcommand\H{\curly H}
\newcommand\I{\curly I}

\renewcommand\O{\mathcal O}
\newcommand\PP{\mathbb P}

\newcommand\C{\mathbb C}

\newcommand\R{\mathbb R}
\newcommand\Z{\mathbb Z}

\newcommand{\Rt}[1]{\stackrel{#1\,}{\longrightarrow}}
\newcommand\To{\longrightarrow}
\newcommand\into{\hookrightarrow}
\newcommand\Into{\ar@{^{ (}->}[r]}

\newfont{\bigtimesfont}{cmsy10 scaled \magstep5}
\newcommand{\bigtimes}{\mathop{\lower0.9ex\hbox{\bigtimesfont\symbol2}}}

\renewcommand\ker{\operatorname{ker}}

\newcommand\Bl{\operatorname{Bl}}
\newcommand\Hom{\operatorname{Hom}}
\renewcommand\hom{\curly H\!om}
\newcommand\Ext{\operatorname{Ext}}
\newcommand\ext{\curly Ext}

\newcommand\beq[1]{\begin{equation}\label{#1}}
\newcommand\eeq{\end{equation}}

\makeatletter \@addtoreset{equation}{section} \makeatother

\newtheorem{thm}[equation]{Theorem}
\newtheorem{lem}[equation]{Lemma}
\newtheorem{cor}[equation]{Corollary}
\newtheorem{prop}[equation]{Proposition}
\newenvironment{rmk}{\noindent\textbf{Remark}.}{\\}
\newenvironment{rmks}{\noindent\textbf{Remarks}.}{\\}

\newtheorem{quest}{Question}

\newtheoremstyle{citing}
  {}
  {}
  {\itshape}
  {}
  {\bfseries}
  {}
  {.5em}
  {\thmnote{#3}}

  {\theoremstyle{citing}
\newtheorem*{custom}{}}

\newcommand{\f}[1]{\Omega_Y^{n-1}({#1}Q)}

\newcommand{\isom}{\cong}
\DeclareMathOperator{\tensor}{\otimes}
\newcommand{\restr}[1]{{\raisebox{-0.3\height}{$\mid_{#1}$}}}

\newcommand{\binomial}[2]{\begin{pmatrix}
                           {#1}\\{#2}
                          \end{pmatrix}
}

\title{Smoothing nodal Calabi-Yau $n$-folds}
\author[S. Rollenske and R. P. Thomas]{S\"onke Rollenske and Richard Thomas}

\begin{document}
\begin{abstract} 
Let $X$ be an $n$-dimensional Calabi-Yau with ordinary double points, where
$n$ is odd.
Friedman showed that for $n=3$ the existence of a smoothing of $X$ implies
a specific type of relation between homology classes
on a resolution of $X$. (The converse is also true, due to work of Friedman,
Kawamata and Tian.)

We sketch a more topological proof of this result, and then extend it to
higher dimensions. For $n>3$ the result is nonlinear; the ``Yukawa product" on the middle dimensional (co)homology plays an unexpected role. We also discuss a converse, proving it for nodal Calabi-Yau hypersurfaces of $\PP^{n+1}$.
\end{abstract}
\maketitle


\section{Introduction}
Fix a Calabi-Yau $n$-fold $X$; for simplicity in this introduction this 
will mean a projective variety with at worst ordinary double point (ODP,
or nodal) singularities and trivial canonical bundle $K_X\cong\O_X$.
Sometimes we will pick a trivialisation
$\Omega\in H^0(K_X)$ of $K_X$ -- a ``complex volume form".

\subsection*{Smooth Calabi-Yaus}
We recall that when $X$ is smooth and three dimensional,
first order deformations of $(X,\Omega)$ are given by its third homology
$H_3(X,\R)$. The form $\Omega$ gives an isomorphism between the first order
deformations $H^1(T_X)$ of $X$ and $H^1(\Omega^2_X)=H^{2,1}(X)$, while the
choice of $\Omega$ adds $H^{3,0}$ to the deformation space. The isomorphism
to $H_3(X,\R)$ is then given by the composition of the natural maps
\beq{H3}
H_3(X,\R)\Rt{\mathrm{PD}} H^3(X,\R) \into H^3(X,\C)\Rt{\mathrm{pr}}H^{3,0}(X)\oplus
H^{2,1}(X).
\eeq

Equivalently, given a 3-cycle on $(X,\Omega)$ we take the $(3,0)+(2,1)$ part
of its Poincar\'e dual
and add this to the period point $[\Omega]\in H^3(X,\C)$ to give the corresponding
first order deformation;
there will then be a unique first order deformation of the complex structure
for which the new class $[\Omega]$ is represented by a holomorphic form of
type $(3,0)$.

More generally deformations of smooth Calabi-Yau $n$-folds are governed by
the $H^{n,0}\oplus H^{n-1,1}$ part of their middle degree (co)homology. Therefore
the Zariski
tangent space to the moduli space of Calabi-Yaus has constant dimension,
showing (by $T^1$-lifting \cite{Ra, Ka}) that the moduli space is smooth
(unobstructed); this is the Bogomolov-Tian-Todorov theorem.

\subsection*{Nodal Calabi-Yaus 3-folds}
Friedman's theorem \cite{Fr} can be seen as an extension of this theory to
nodal Calabi-Yau 3-folds. That
is, fixing a Calabi-Yau 3-fold $X$ with at worst ODPs, one can interpret
his result as saying that first order deformations of $(X,\Omega)$ are still
isomorphic to $H_3(X,\R)$. (Kawamata \cite{Ka} and Tian \cite{Ti} later showed
-- in any dimension $n$ -- that again all first order deformations are unobstructed
and the moduli space is smooth.)

Friedman stated his theorem as follows. There exist \emph{small resolutions}
of $X$ given by replacing each ODP $p_i$ with a smooth rational curve $C_i$
with normal bundle $\O_{\PP^1}(-1)\oplus\O_{\PP^1}(-1)$. Different small
resolutions differ by flops of the curves $C_i$ and need not be projective
or even K\"ahler; we fix one $X^+$\!.

\begin{thm} \cite{Fr} \label{Fr1}
$X$ admits a first order smoothing if and only if there is a relation
$$
\sum\delta_i[C_i]\ =\ 0 \quad\text{in}\ H_2(X^+\!,\R)
$$
with each $\delta_i\ne0$.
\end{thm}

Since small resolutions are special to three dimensions we rephrase this
in terms of the blow up $Y$ of $X$ at the ODPs $p_i$. Each exceptional divisor
$Q_i$ is a 2-dimensional quadric $\PP^1\times\PP^1$ (the small resolutions
come from blowing down one of the two rulings). Let $A_i$ and $B_i$ denote
the homology classes of $\PP^1\times\{\mathrm{pt}\}$ and $\{\mathrm{pt}\}\times\PP^1$.
Combined with the unobstructedness results mentioned above, Friedman's theorem becomes the following.

\begin{thm} \label{Fr2}
$X$ admits a smoothing if and only if there is a relation
$$
\sum\delta_i(A_i-B_i)\ =\ 0 \quad\text{in}\ H_2(Y,\R)
$$
with each $\delta_i\ne0$.
\end{thm}

(Notice that the classes $A_i$ and $B_i$ are nonzero in $H_2(Y)$, since 
each has intersection $-1$ with $[Q_i]$.)

By some elementary topology (excision, long exact sequence of a pair) we
can interpret this as follows. The relation in homology is a 3-chain $\overline{\Delta}$
whose boundary may be taken to be a union of $S^2$s -- multiples $\delta_i$
of the $C_i$
in $X^+$\! \eqref{Fr1}, or of the anti-diagonals in $Q_i\cong\PP^1\times\PP^1$
in $Y$ \eqref{Fr2}. On blowing down to $X$ the $S^2$s are collapsed to a
point and $\overline{\Delta}$ projects to a 3-cycle $\Delta$. This lifts
on a topological model $\widetilde X$ of the smoothing to a cycle $\widetilde{\Delta}$
whose intersection with the vanishing cycle at $p_i$ is also $\delta_i$ (see
\cite{STY} for example).
(There is a sign issue here; changing the orientation on the vanishing cycle
corresponds to flopping the curve $C_i\subset X^+\!$ or swapping the cycles
$A_i$ and $B_i$ in $Y$ \cite{STY}.)

On this model $\widetilde X$ the holomorphic $3$-form of $X$ pulls back to
a $3$-form which is degenerate precisely along the vanishing cycles. Adding small multiples of the 3-cycle $\widetilde\Delta$ to the
$3$-form (in the way described in \eqref{H3}) breaks this degeneracy because
its pairing with the vanishing cycle at $p_i$ is $\delta_i\ne0$. (We think
of adding $\widetilde\Delta$ to the $(3,0$) form as ``inflating" the vanishing
cycle.) The result is a path of cohomology classes on $\widetilde X$ beginning at the pullback of the original $(3,0)$-form. There is a corresponding unique path of complex structures on $\widetilde X$ for which these classes are of type $(3,0)$.

Thus we see that deformations of $X$ again correspond to $H_3(X,\R)$. A given
class $[\Delta]\in H_3(X,\R)$ corresponds to a deformation which
smooths the ODP $p_i$ if and only if the intersection of $[\Delta]$ and $p_i$
is nonzero. Here we are defining the intersection number $\delta_i$ of $[\Delta]$
and $p_i$ by any of the following equivalent prescriptions.
\begin{itemize} \item
Lift $\Delta$ to a 3-cycle $\widetilde{\Delta}$
on a topological model $\widetilde X$ of the smoothing and intersect with
the vanishing cycle at $p_i$,
\item Lift $\Delta$ to a chain $\overline{\Delta}$ on a small resolution
$X^+$\!, then $\partial\overline{\Delta}$ contains a multiple $\delta_i$ of
the exceptional curve $[C_i]$,
\item Lift $\Delta$ to a chain $\overline{\Delta}$ on the blowup $Y$, then
$\partial\overline{\Delta}$ contains a multiple $\delta_i$ of the cycle $A_i-B_i$
on $Q_i$.
\end{itemize}

In Section \ref{3} we sketch a short proof of Friedman's theorem along such
topological lines, using logarithmic forms and residues to play the role
of chains and boundaries (to which they are Poincar\'e dual).

\subsection*{Nodal Calabi-Yau $n$-folds}
Let $X$ be an odd dimensional Calabi-Yau $n$-fold with ODPs $p_i\in X$.
Let $Y$ denote the blow up of $X$ in the $p_i$, with exceptional divisors
$Q_i$. These are $(n-1)$-dimensional quadrics with standard $(n-1)/2$-dimensional
planes whose homology classes we denote $A_i,B_i$; see \eqref{AB}.

Extending Friedman's results to higher dimensions is complicated by the
other Hodge summands in the middle degree (co)homology. For instance the
$A_i-B_i$ classes on $Y$ are of pure Hodge type $(\frac{n+1}2,\frac{n+1}2)$.
For a linear combination of them to be zero in $H^{n+1}(Y)$ as in Theorem
\ref{Fr2} would require it to be $\bar\partial$ of a form of type $(\frac{n+1}2,
\frac{n-1}2)$. Only in $n=3$ dimensions is this a form of type $(n-1,1)$
(i.e. a form of the type that controls deformations of complex structure).
So the relation must be more complicated when $n>3$.

The best way of explaining the difference between three and higher
dimensions is as follows. The local model of the smoothing
\[
X_{\epsilon}=\left\{f_\epsilon:=\sum_{i=1}^{n+1}x_i^2\,-\epsilon=0\right\}\subset\C^{n+1}
\]
has the natural parameter (or modulus) $\epsilon$. In fact this generates
the sheaf of local smoothings $\ext^1(\Omega_{X_0},\O_{X_0})$;
see Lemma \ref{ext}.

However there is another natural
modulus given by the pairing of the holomorphic $n$-form with the vanishing
cycle (the ``complex volume" of the vanishing cycle). A local model of the$(n,0)$-form
is given by
\[
\Omega_{f_\epsilon}:=\frac{dx_1\ldots dx_n}{\partial f_\epsilon/\partial
x_{n+1}}=\frac{dx_1\ldots dx_n}{2x_{n+1}}
\]
on $X_\epsilon$. (This is the Poincar\'e residue of the form $dx_1\ldots
dx_{n+1}/f_\epsilon$ on $\C^{n+1}$ with first order poles along $X_\epsilon$.)

The vanishing cycle $L_{\epsilon}$ is the $n$-sphere real slice of $X_\epsilon$
given by
taking those points with $x_i\in\sqrt\epsilon\,\R$ for all $i$. The integral
of $\Omega_{f_\epsilon}$ over this can be computed \cite{Ott} to be
\beq{integ}
\int_{L_\epsilon}\Omega_{f_\epsilon}=O\big(\epsilon^{\frac{n-1}2}\big).
\eeq
(In even dimensions the vanishing cycle has monodromy $[L_\epsilon]\mapsto
-[L_\epsilon]$ as $\epsilon$ circles
the origin, which explains the ambiguity in the square root.) 
Therefore the two moduli are the same only in 3 dimensions; in general
one has to \emph{take the $(n-1)/2$\,th power of the complex structure smoothing
modulus to get the modulus related to homology}. Another way of saying this
is that the homological modulus does not correspond to a first order deformation
of complex structure (except in dimension 3) but an $(n-1)/2$\,th order deformation.

Therefore, working entirely algebraically, we take the $(n-1)/2$\,th power
of any first order deformation via the
\emph{Yukawa product}\footnote{On a smooth Calabi-Yau $n$-fold the (B-model)
Yukawa product on its middle dimensional (co)homology,
$$
H^{n-i,i}(X)\otimes H^{n-j,\,j}(X)\to H^{n-(i+j),\,i+j}(X),
$$
is given by the isomorphism $\Lambda^iT_X\cong\Omega_X^{n-i}$ (induced by
the complex volume form $\Omega$) and the cup product $H^i(\Lambda^iT_X)\otimes
H^j(\Lambda^jT_X)\to H^{i+j}(\Lambda^{i+j}T_X)$. In particular in odd dimensions
the $(n-1)/2$\,th power of a class in $H^1(T_X)$ lies in
$H^{\frac{n+1}2,\frac{n-1}2}(X)$.}, since this gives the ($(\frac{n+1}2,\frac{n-1}2)$
Hodge part of the) $(n-1)/2$\,th
derivative of the period point of $X$ under a path of smoothings with
the given first order derivative. In Section \ref{smoohom} we show that
this gives rise to homology relations amongst
the $A-B$ classes from smoothings that simultaneously
smooth each of the ODPs: 

\begin{custom}[Theorem \ref{biggy}.]
Fix a first order deformation $e\in\Ext^1(\Omega_X,\O_X)$.
Let $\oplus_i\epsilon_i$ be the image of $e$ under the global-to-local map
to $H^0(\ext^1(\Omega_X,\O_X))\cong H^0(\oplus_i\O_{p_i})$. Then we have the relation
$$
\sum\delta_i(A_i-B_i)=0\ \in H_{n-1}(Y),
$$
where $\delta_i=\epsilon_i^{\frac{n-1}2}$ and $e$ gives rise to a smoothing if and only if all $\delta_i$ are nonzero.
\end{custom}

\begin{rmks}
Our choice of generator of $\ext^1(\Omega_X,\O_X)\cong\oplus_i\O_{p_i}$ at each double point is explained in Section \ref{smoohom}. If there are $k$ ODPs then this gives an isomorphism $H^0(\ext^1(\Omega_X,\O_X))\cong\C^k$ and natural coordinates $\epsilon_1,\ldots,\epsilon_k$ thereon. In these coordinates, then, the image of the (linear!) global to local map $\Ext^1(\Omega_X, \O_X)\to H^0(\ext^1(\Omega_X,\O_X))$ is  contained in 
\[\left\{ (\epsilon_1,\ldots,\epsilon_k)\colon \sum_i \epsilon_i^{\frac{n-1}2} (A_i-B_i)=0 \text{ in } H_{n-1}(Y)\right\}.\]
In other words the linear subspace of local infinitesimal first order deformations that are realised globally is contained in an intersection of degree $\frac{n-1}2$ Fermat hypersurfaces in $H^0(\ext^1(\Omega_X,\O_X))$.
We are very grateful to the referee for this point of view on Theorem \ref{biggy}; it seems very rare to have such nonlinear constraints on either a (linear) first order deformation problem or a (linear) spectral sequence.
\end{rmks}

The analysis \eqref{integ} also suggests a differential geometric approach
to proving Theorem
\ref{biggy}. The idea would be to show that in a smoothing the vanishing
cycles $L_i$ smoothing the ordinary double points $p_i$ admit representatives
which are \emph{special Lagrangian}. This is widely expected to be true
(the local model above is indeed special Lagrangian)
but has not yet been proved.

It would follow that their pairings with the cohomology
class of the holomorphic $n$-form on the smoothing are all \emph{nonzero}
complex numbers $\delta_i$. The Poincar\'e dual of the holomorphic $n$-form
is an $n$-cycle (with complex coefficients) $\widetilde\Delta$ which deforms
naturally to an $n$-cycle $\Delta$ in the original Calabi-Yau with ODPs as
we degenerate back from the smoothing to the central fibre. Lifting to the
resolution $Y$ produces a chain $\overline\Delta$ whose boundary is
$\sum_i\delta_i(A_i-B_i)$. \smallskip

Conversely one might wonder whether all such homology relations arise from
smoothings. We formulate a precise question in Section \ref{hyper} and prove
it for nodal hypersurfaces of projective space. (This does \emph{not} mean
we are confident that it holds in general, however.)

\subsection*{Even dimensions}
There is no analogue of Theorem \ref{biggy} since odd dimensional
quadrics have no middle dimensional homology and the homology of $Y$ is simply
the direct sum of that of $X$ and the exceptional divisors $Q_i$. Anyway
when $n$ is even there is already a cycle on any smoothing $\widetilde X$
which intersects
the vanishing cycle nontrivially (thus playing the role of $\widetilde\Delta$
above in odd dimensions) -- the vanishing cycle itself. Therefore if the
converse
formulated in Section \ref{hyper} turned out to be true
then we would expect
\emph{any} nodal even-dimensional Calabi-Yau to have a smoothing, without
any conditions.

\subsection*{Acknowledgements} We thank Samuel Stark for helping with the computation of the torsion in Lemma \ref{piO}, and Lev Borisov, Mark Gross,
James Otterson and Dmitri Panov for useful conversations. In particular the simplified proof of Friedman's theorem sketched in Section \ref{3} was worked out jointly
with Otterson as part of his Imperial College PhD \cite{Ott}. The referee also made some suggestions which greatly improved the paper.

The first author was supported by a Forschungsstipendium of the Deutsche
Forschungsgemeinschaft (DFG). The second author was partially supported by
a Royal Society university research fellowship.

\section{Set up} \label{3}

We first fix some notation. Let $X$ be an $n$-dimensional compact analytic
space with trivial canonical bundle $K_X=\O_X$  whose only singularities
are ordinary double points $p_i\in X$. Denoting by $\pi\colon Y\to X$ the
blow-up of the double points, we assume further that $Y$ is K\"ahler.
The exceptional divisor of $\pi$
is a disjoint union of quadrics $Q=\coprod_iQ_i$
with normal bundle $\O_Q(Q)=\O_Q(-1)$ and $K_Y=\O_Y((n-2)Q)$.
So long as $X$ has at least one ODP then
$H^1(X,\O_X)=0$ by \cite[Theorem 8.3]{Ka85} so, as remarked by Friedman, every smoothing will again have trivial canonical bundle and is a 
smooth Calabi-Yau provided it is K\"ahler.

We denote by $T_Y$ the tangent sheaf of $Y$. We often use the isomorphism
\beq{cot-tan}
\Omega^k_Y\isom \Lambda^{n-k}T_Y\tensor K_Y=\Lambda^{n-k}T_Y((n-2)Q).
\eeq

\begin{lem} \label{ext}
$\ext^i(\Omega_X,\O_X)$ is nonzero only for $i=0$ and $i=1$. For $i=0$ we
denote the sheaf $\hom(\Omega_X,\O_X)$ by $T_X$. For $i=1$ we have
\[
\ext^1(\Omega_X,\O_X)\cong\oplus_i\O_{p_i}.
\]
\end{lem}

\begin{proof}
Locally, around a double point, $X$ is isomorphic to a hypersurface in $\C^{n+1}$
described
by the equation $f=\sum_i x_i^2$. Thus we have a locally-free resolution
\begin{equation} \label{local}
0\to\O_X\Rt{df}\Omega_{\C^{n+1}}\restr X\to\Omega_X\to0.
\end{equation}
Dualising yields
$$
0\to T_X\to T_{\C^{n+1}}\restr X\Rt{df}\O_X\to\ext^1(\Omega_X,\O_X)\to0,
$$
with vanishing higher $\ext$s. The central map is $$\xymatrix{
\O_X^{\oplus(n+1)}\ar[rr]^(.55){(2x_1,\ldots,2x_n)} && \O_X}$$
with cokernel the structure sheaf of the origin.
\end{proof}

Therefore the local-to-global spectral sequence $H^i(\ext^j)\Rightarrow\Ext^{i+j}$
collapses to the long exact sequence
\beq{loctoglob}
0\to H^1(T_X)\to\Ext^1(\Omega_X,\O_X)\to\bigoplus_iH^0(\O_{p_i})
\to H^2(T_X).
\eeq
The second term governs the first order deformations of $X$, while
the first describes the equisingular ones: those that do not smooth the ordinary
double points. The third term looks locally about each
$p_i$ and compares a global deformation to the local universal deformation
$\sum_jx_j^2=\epsilon$.
The last arrow is the obstruction to finding a global first order deformation matching a given local one about each $p_i$.

\begin{lem} \label{piO} $L_j\pi^*\Omega_X=0$ for $j\ge1$; for $j=0$ we have the exact sequence
$$
0\to\O_Q(-1)\to\pi^*\Omega_X\to\Omega_Y(\log Q)(-Q)\to0.
$$
\end{lem}

\begin{proof}
We work locally on the blow up of the affine ODP $\{\sum x_i^2=0\}\subset\C^{n+1}$.
This is the subvariety of $\C^{n+1}\times\PP^n$ defined by the equations $x_iX_j=x_jX_i$ and $\sum X_i^2=0$, where the $X_i$ are the standard homogeneous coordinates on $\PP^n$. 

We start by working out the image of $(D\pi)^*\colon\pi^*\Omega_X\to\Omega_Y$ in the patch $X_{n+1}=1$ without loss of generality. This has coordinates $X_1,\ldots,X_n$ and $x_{n+1}$ (the others being determined by the relations $x_i=x_{n+1}X_i $) subject to the relation $X_1^2+\ldots X_n^2+1=0$. The pullbacks $\bar x_i:=x_i\circ\pi$ of the coordinates $x_i$ are $\bar x_i=X_ix_{n+1}$ and $\bar x_{n+1}=x_{n+1}$. Since $\Omega_X$ is locally generated by $dx_i$ ($i=1,\ldots n$) and $dx_{n+1}$, it follows that im\,$(D\pi)^*$ is locally generated by $d\bar x_i=X_idx_{n+1}+x_{n+1}dX_i$ and $dx_{n+1}$. It is therefore generated by $dx_{n+1}$ and $x_{n+1}dX_i,\ i=1,\ldots,n$.

The exceptional divisor $Q$ is described by $x_{n+1}=0$, so im\,$(D\pi)^*$ is locally generated by $dx_{n+1}/x_{n+1}$ and $dX_i$. But these are the generators of $\Omega_Y(\log Q)$, by definition, so im\,$(D\pi)^*=\Omega_Y(\log Q)(-Q)$.\medskip

Since $\Omega_Y$ is locally free and $(D\pi)^*$ is an isomorphism away from $Q$, its kernel is the torsion subsheaf of $\pi^*\Omega_X$. So we are left with checking this is $\O_Q(-1)$. For this we use the commutative diagram of exact sequences
$$\xymatrix@R=18pt@C=25pt{
0 \rto& \pi^*\O_X \dto_{s_Q^{-2}}\rto^-{d\sum x_i^2}& \pi^*\!\left(\Omega_{\C^{n+1}}\big|_X\right) \dto^{(Dp)^*}\rto& \pi^*\Omega_X \dto^{(D\pi)^*}\rto&0 \\
0 \rto& \O_Y(-2) \rto^-{d\sum X_i^2}& \Omega_{\Bl_0\C^{n+1}} \big|_Y\rto& \Omega_Y \rto& 0.\!}
$$
The first is the exact sequence of K\"ahler differentials \eqref{local} for the hypersurface $X=(f=0)\subset\C^n$, pulled back by $\pi^*$. (Since \eqref{local} defines a locally-free resolution of $\Omega_X$, and the pullback sequence is still exact, it follows that $L_j\pi^*\Omega_X=0$ for $j\ge1$.)
The second is the exact sequence of K\"ahler differentials for the hypersurface $Y=(\sum X_i^2=0)$ inside the blow-up $p\colon\Bl_0\C^{n+1}\to\C^{n+1}$. We use the fact that $\sum X_i^2=\sum x_i^2/s_E^2$ is a section of (the pullback from $\PP^n$ of) $\O(2)$. Here $s_E$ cuts out the exceptional divisor $E\cong\PP^n\subset\Bl_0\C^{n+1}$, restricting over $Y$ to $s_Q$ cutting out $Q\subset Y$.

Thinking of the diagram as a short exact sequence of (vertical) two-term complexes, the corresponding long exact of cohomologies gives
$$
0\to\ker(D\pi)^*\to\O_{2Q}(-2)\to\Omega_E|\_Q\to\Omega_Q\to0,
$$
where $2Q\subset Y$ is the scheme-theoretic doubling of the divisor $Q$. This sequence is the concatenation of the exact sequence
\begin{equation}\label{Dpi}
0\to\ker(D\pi)^*\to\O_{2Q}(-2)\to\O_{Q}(-2)\to0
\end{equation}
and the exact sequence of K\"ahler differentials for the hypersurface $Q\subset E$ (cut out by the section $\sum X_i^2$ of $\O_E(2)$),
$$
0\to\O_Q(-2)\to\Omega_E|\_Q\to\Omega_Q\to0.
$$
By \eqref{Dpi}, $\ker(D\pi)^*$ is therefore the twist by $\O(-2)$ of the ideal sheaf $\O_Q(1)$ of $Q\subset 2Q$.
\end{proof}

\begin{prop} \label{ott}
$R\hom(\Omega_X,\O_X)=R\pi_*\Omega^{n-1}_Y(\log Q)$.
\end{prop}

\begin{proof}
Notice that there is a perfect pairing of vector bundles $\Omega^{n-1}_Y(\log
Q)\otimes\Omega^1_Y(\log Q)(-Q)\to K_Y$ given by wedging differential forms.
Therefore
\begin{eqnarray*}
R\pi_*\Omega^{n-1}_Y(\log Q) &\cong& R\pi_*\hom(\Omega^1_Y(\log Q)(-Q),K_Y)
\\ &\xrightarrow\sim& R\pi_*R\hom(L\pi^*\Omega_X,K_Y) \\
&\cong& R\hom(\Omega_X,R\pi_*K_Y).
\end{eqnarray*}
Here the arrow, given by Lemma \ref{piO}, is an isomorphism because
its cone
$$
R\pi_*R\hom(\O_Q(-1),K_Y)\cong R\pi_*(K_Y|_Q[-1])\cong R\pi_*\O_Q(2-n)[-1]
$$
vanishes for $n\ge3$.

It remains to prove that $R\pi_*K_Y\cong\O_X$. By relative Serre duality
down the (relative dimension 0) projective morphism $\pi$, $R\pi_*K_Y$ is dual to
$(R\pi_*\O_Y)\otimes K_X$. But $K_X\isom\O_X$ and the ordinary double point is a rational singularity so $R\pi_*\O_Y\isom\O_X$.
\end{proof}

\section{Dimension three}
We first consider the case $n=3$ treated by Friedman, proving the variant
Theorem \ref{Fr2} of his result. Our treatment will
be very brief; the details can be worked out in much the same way as the
higher dimensional case in the next section.

By Proposition \ref{ott} the deformation space $\Ext^1(\Omega_X,\O_X)$ of
$X$ is simply $H^1(\Omega^2_Y(\log Q))$. There is a standard exact sequence
\beq{exres}
0\to\Omega^2_Y\to\Omega^2_Y(\log Q)\to\Omega^1_Q\to0,
\eeq
with the last map the residue map. We obtain the long exact sequence
$$
0\to H^1(\Omega^2_Y)\to H^1(\Omega^2_Y(\log Q))\to H^1(\Omega^1_Q)\to H^2(\Omega^2_Y).
$$
The last arrow is the Gysin map $H^{1,1}(Q)\to H^{2,2}(Y)$. This is an injection
on the $A_i+B_i$ classes since their intersection with $[Q_i]\in
H^{1,1}(Y)$ is $-2$. Dividing out by these classes gives
\beq{frex}
0\to H^1(\Omega^2_Y)\to\Ext^1(\Omega_X,\O_X)\to\bigoplus_i\langle A_i-B_i\rangle
\to\frac{H^{2,2}(Y)}{\bigoplus_i\langle A_i+B_i\rangle}\,.
\eeq
It is easy to show that in
fact this is precisely the exact sequence \eqref{loctoglob} for $n=3$; something
very similar will be shown in Corollary \ref{coho}
in the next section for dimensions $n>3$. As we will see there, the point
is the following: by Lemma \ref{ext} and Proposition \ref{ott}, the complex
$R\pi_*\Omega^2_Y(\log Q)$ has cohomology sheaves only in degrees 0 and 1.
And on applying $R\pi_*$ to the exact sequence \eqref{exres} the
first term provides all of the 0th cohomology of $R\pi_*\Omega^2_Y(\log Q)$,
and the third term provides all of the 1st cohomology. For $n=3$ this amounts
to
$$
\qquad R^0\pi_*\Omega^2_Y=
R^0\pi_*\Omega^2_Y(\log Q), \qquad R^{\ge1}\pi_*\Omega^2_Y=0,
$$
and
$$
\frac{R^1\pi_*\Omega_Q}{\bigoplus_i\langle A_i+B_i\rangle}=
R^1\pi_*\Omega^2_Y(\log Q), \qquad R^{\ne1}\pi_*\Omega_Q=0.
$$
Now a first order smoothing of $X$ is a class in
$\Ext^1(\Omega_X,\O_X)$ which maps to a nonzero first order smoothing of the ODP $p_i$
in the sequence \eqref{loctoglob} for each $i$. Therefore it is a class which maps to
a nonzero multiple $\delta_i$
of $A_i-B_i$ in \eqref{frex} for each $i$. By the exactness of \eqref{frex} such a class
exists if and only if $\sum_i\delta_i(A_i-B_i)$ is zero
in $H^{2,2}(Y)\big/\oplus_i\langle A_i+B_i\rangle$. This is equivalent to
$\sum_i\delta_i(A_i-B_i)$ being zero in the isomorphic group which is 
the kernel in $H^{2,2}(Y)$ of cupping with all of the PD$[Q_i]$ classes.
Therefore it is equivalent to $\sum_i\delta_i(A_i-B_i)$ being zero
in $H^{2,2}(Y)$.
And by \cite{Ka, Ti} any first order smoothing can be realised as
the first derivative of an actual smoothing.

\section{From smoothings to homology}\label{smoohom}
We assume $n\ge5$ from now on. (In fact $n\ge3$ works similarly with minor
modifications.) We first show that in the isomorphism
of Proposition \ref{ott}, we can pass to the subsheaf $\Omega^{n-1}_Y(\log
Q)(-(n-3)Q)$ of $\Omega^{n-1}_Y(\log Q)$.

\begin{prop} \label{Rpi1} Fix $n\geq 5$. The inclusions
\[
\Omega_Y^{n-1}(\log Q)(-(n-3)Q)\,\subseteq\ \Omega_Y^{n-1}(-(n-4)Q)
\,\subseteq\ \Omega^{n-1}_Y(\log Q)
\]
combined with Proposition \ref{ott} induce isomorphisms
\begin{equation} \label{big} 
\begin{split}
R\hom(\Omega_X,\O_X) &\isom R^{\leq 2}\pi_*\big(\Omega^{n-1}_Y(-(n-4)Q)\big)
\\ &\isom R^{\leq 2}\pi_*\big(\Omega^{n-1}_Y(\log Q)(-(n-3)Q)\big).
\end{split}
\end{equation}
\end{prop}

\begin{proof}
By Proposition \ref{ott} and Lemma \ref{ext},
\[
R\hom(\Omega_X,\O_X)=R^{\le2}\pi_*\Omega^{n-1}_Y(\log Q).
\]

The inclusions
\[
\Omega^{n-1}_Y(\log Q)(-(i+1)Q)\,\subseteq\ \f{-i}\,\subseteq\ 
\Omega^{n-1}_Y(\log Q)(-iQ)
\]
have cokernels $\Omega^{n-1}_Q(i)$ and $\Omega^{n-2}_Q(i)$ respectively.
These have vanishing $H^{\le2}$
for $i=n-4,n-5,\ldots,0$ by Proposition \ref{cohomQ} in the Appendix.
Therefore they induce isomorphisms
on $R^{\le2}\pi_*$, proving the result inductively.
\end{proof}

\begin{prop} \label{Rpi2} Fix $n\ge5$.
The inclusion $T_Y\isom\,\Omega_Y^{n-1}(-(n-2)Q)\subseteq\,\Omega_Y^{n-1}(-(n-3)Q)$
induces isomorphisms
\begin{itemize}
\item $\pi_*T_Y\cong\pi_*\Omega^{n-1}_Y(-(n-3)Q)\cong\pi_*\Omega^{n-1}_Y(\log
Q)\cong T_X$,
\item $R^i\pi_*T_Y=0=R^i\pi_*\Omega^{n-1}_Y(-(n-3)Q)$ for $i=1,2$.
\end{itemize}
\end{prop}

\begin{proof}
The inclusions
\[
\Omega^{n-1}_Y(-(n-2)Q)\,\subseteq\ \Omega^{n-1}_Y(\log Q)(-(n-2)Q)\,\subseteq\
\Omega^{n-1}_Y(-(n-3)Q)
\]
have cokernels $\Omega^{n-2}_Q(n-2)$ and $\Omega^{n-1}_Q(n-3)$ respectively.
By Proposition \ref{cohomQ} in the Appendix these have vanishing
$H^{\le2}$, so they induce isomorphisms on $R^{\le2}\pi_*$.

The inclusion 
\[
\Omega^{n-1}_Y(-(n-3)Q)\,\subseteq\ \Omega^{n-1}_Y(\log Q)(-(n-3)Q)
\]
has cokernel $\Omega^{n-2}_Q(n-3)$, which has no $H^0$ by Proposition \ref{cohomQ}.
It therefore induces an isomorphism on $\pi_*$, which by Proposition \ref{Rpi1}
and Lemma \ref{ext} gives the first sequence of isomorphisms.

Since $R^1\pi_*T_Y$ and $R^2\pi_*T_Y$ are supported on the double points
we can calculate them in the local model $Y=\O_Q(-1)$ and $X=\{\sum x_i^2=0\}\subset \C^{n+1}$. Since  $\O_Q(-1)$ is
rigid and $X$ is affine it follows that $R^1\pi_*T_Y=0$. We therefore get
the exact sequence
\[ 0 \to R^1\pi_*\Omega_Y^{n-1}(\log Q)(-(n-3)Q)\to R^1\pi_*\Omega^{n-2}_Q(n-3)\to
R^2\pi_*T_Y\to 0.\]
The first term is isomorphic to $\O_0$ by Proposition \ref{Rpi1} and Lemma
\ref{ext}. The second term is also $\O_0$, by Proposition \ref{cohomQ}. Therefore
$R^2\pi_*T_Y=0$.
\end{proof}

\begin{cor} \label{coho}
Taking sheaf cohomology of the exact sequence
\beq{res}
0\to\Omega^{n-1}_Y(-(n-3)Q)\to\Omega^{n-1}_Y(\log Q)(-(n-3)Q)\Rt{\mathrm{Res}}
\Omega^{n-2}_Q(-(n-3)Q)\to0
\eeq
induces the sequence \eqref{loctoglob}.
\end{cor}

\begin{proof}
Let $E$ denote the complex
\[
E=R\hom(\Omega_X,\O_X)=R^{\leq 2}\pi_*\big(\Omega^{n-1}_Y(\log Q)(-(n-3)Q)\big).
\]
By Lemma \ref{ext} this has cohomology sheaves $\H^0(E)=T_X$
and $\H^1(E)=\oplus_i\O_{p_i}$
only in degrees 0 and 1. The sequence \eqref{loctoglob} arises by taking
$H^1$ and $H^2$ of the tautological exact triangle
\beq{hh}
\H^0(E)\to E\to \H^1(E)[-1].
\eeq
By Proposition \ref{Rpi2} the $\H^0(E)$ part comes entirely
from the first term of \eqref{res}. That is, the inclusion
$\Omega_Y^{n-1}(-(n-3)Q)\subseteq
\Omega_Y^{n-1}(\log Q)(-(n-3)Q)$ induces an isomorphism
\[
R^{\le2}\pi_*\Omega_Y^{n-1}(-(n-3)Q)\Rt{\simeq}\H^0(E).
\]
Therefore by \eqref{res} the residue map
$\Omega_Y^{n-1}(\log Q)(-(n-3)Q)\to\Omega_Q^{n-2}(n-3)$
induces an isomorphism on $R^1\pi_*$, while $R^0\pi_*$ and $R^2\pi_*$ vanish
on $\Omega_Q^{n-2}(n-3)$ by Proposition \ref{cohomQ}. 
Since $R^1\pi_*\big(\Omega^{n-1}_Y(\log Q)(-(n-3)Q)\big)=\H^1(E)$ this says
that the residue map induces an isomorphism
\[
\H^1(E)[-1]\Rt{\simeq}R^{\le2}\pi_*\Omega_Q^{n-2}(n-3).
\]
That is, the $\H^1(E)[-1]$ part of $E$ all comes from the third term of \eqref{res}.

Therefore applying $R^{\le2}\pi_*$ to \eqref{res} gives an exact triangle
\begin{multline*}
R^{\le2}\pi_*\big(\Omega^{n-1}_Y(-(n-3)Q)\big)\to R^{\le2}\pi_*\big(\Omega^{n-1}_Y(\log
Q)(-(n-3)Q)\big) \\ \to R^{\le2}\pi_*\big(\Omega^{n-2}_Q(-(n-3)Q)\big),
\end{multline*}
and this is exactly \eqref{hh}. So applying
$H^i_Y=H^i_X(R\pi_*)=H^i_X(R^{\le2}\pi_*)$ (for $i=1,2$) to the sequence
\eqref{res} gives \eqref{loctoglob}.
\end{proof}

\begin{rmk}
By Proposition \ref{Rpi2}, $H^i(\Omega^{n-1}_Y(-(n-3)Q))=H^i(T_Y)$ for $i\le2$.
Therefore \eqref{loctoglob}, the cohomology exact sequence of
\eqref{res}, can also be written
\[
0\to H^1(T_Y)\to\Ext^1(\Omega_X,\O_X)\to\bigoplus_iH^0(\O_{p_i})\to H^2(T_Y)\to\dots
\]
That is, the equisingular deformations of $X$ correspond to deformations
of the resolution $Y$; a well known fact in more general situations.

From Proposition \ref{Rpi1} we further deduce that
\[\Ext^1(\Omega_X, \O_X)\isom H^1(\Omega_Y^{n-1}(\log Q))\isom H^1(\Omega^{n-1}_Y)=H^{n-1,1}(Y),\]
so deformations of $X$ are controlled by the topology of $Y$. This is the
origin of the unobstructedness result of \cite{Ka}: any
smoothing $X_t$ of $X$ has deformation space $H^{n-1,1}(X_t)$ which can
be shown by mixed Hodge structures to be of the same dimension
as $H^{n-1,1}(Y)$. Therefore $T^1$-lifting applies.
\end{rmk}

\subsection{The Yukawa product}
Now fix $n=2m+1\geq 5$ to be odd. The isomorphisms $\Omega_Y^{n-1}(-(n-4)Q)\cong T_Y(2Q)$ and $\Omega^{m+1}_Y\cong\Lambda^mT_Y((n-2)Q)$ of \eqref{cot-tan} induce
\beq{wedgeiso}
\bigwedge{\!\!}^m\big(\Omega_Y^{n-1}(-(n-4)Q)\big)\ \cong\ 
\bigwedge{\!\!}^m\big(T_Y(2Q)\big)\ \cong\ \Omega^{m+1}_Y(Q).
\eeq
\begin{lem} \label{lemwedge}
Under the above isomorphism \eqref{wedgeiso}, the subsheaf
\beq{incl}
\bigwedge{\!\!}^m\big(\Omega_Y^{n-1}(\log Q)(-(n-3)Q)\big)\ \subseteq\ 
\bigwedge{\!\!}^m\big(\Omega_Y^{n-1}(-(n-4)Q)\big)
\eeq
is  $\Omega^{m+1}_Y(\log Q)\subseteq\,\Omega^{m+1}_Y(Q)$.
\end{lem}

\begin{proof}
Define $K$ be the kernel of $T_Y(2Q)\to N_Q(2Q)$, where $N_Q=\O_Q(-1)$ is the normal bundle to $Q$. Then under the isomorphism
\eqref{cot-tan} the inclusion $\Omega_Y^{n-1}(\log Q)(-(n-3)Q)
\subseteq\Omega_Y^{n-1}(-(n-4)Q)$ becomes $K\subseteq T_Y(2Q)$.

Wedging $m$ times shows that \eqref{incl} is isomorphic to
$\bigwedge^{\!m\!}K\subseteq\bigwedge^{\!m\!}\big(T_Y(2Q)\big)$, i.e.
the kernel of $\bigwedge^{\!m\!}\big(T_Y(2Q)\big)\to
\bigwedge^{\!m-1\!}\big(T_Y(2Q)\big)\restr Q\otimes N_Q(2Q)$. Under the final isomorphism of \eqref{wedgeiso} this is the kernel of $\Omega^{m+1}_Y(Q)\to
\Omega^{m+1}_Q(Q)$. But this is $\Omega^{m+1}_Y(\log Q)$, as claimed.
\end{proof}

Next we use the Poincar\'e residue maps $\Omega^k_Y(\log Q)\Rt{\mathrm{Res}}\Omega^{k-1}_Q$.
These factor through the restriction of $\Omega^k_Y(\log Q)$ to $Q$, followed by the quotient map in the exact sequence $0\to\Omega^k_Q\to\Omega^k_Y(\log Q)\restr Q\to\Omega^{k-1}_Q\to0$ of bundles on $Q$.

Combined with the above isomorphism of Lemma \ref{lemwedge} we get the following Lemma. The right hand vertical map is given by a similar construction to the others of this section, but on $Q$ instead of $Y$. That is, we use the canonical bundle of $Q$ to identify $\Omega^{n-2}_Q(n-3)$ with $T_Q(-2)$, wedge $m$ times and then identify $(\Lambda^mT_Q)(1-n)$ with $\Omega^m_Q$.

\begin{lem} The following diagram is commutative,
$$
\spreaddiagramrows{-.5pc}
\spreaddiagramcolumns{1pc}
\xymatrix{
\bigwedge{\!\!}^m\big(\Omega^{n-1}_Y(\log Q)(-(n-3)Q)\big)
\ar[d]^\sim\ar[r]^(.6){\bigwedge^{\!m\!}\mathrm{Res}\ }
& \bigwedge{\!\!}^m\big(\Omega^{n-2}_Q(n-3)\big) \ar[d]^\sim \\
\Omega^{m+1}_Y(\log Q) \ar[r]\ar[r]^(.6){\mathrm{Res}} & \Omega^{m}_Q.}
$$
\end{lem}

\begin{proof}
For bundles, wedging and restriction to $Q$ commute, and as remarked above the Poincar\'e residue maps factor through restriction to $Q$. Therefore the above claim is a statement that can be checked on $Q$, where it follows from the following standard duality of the Koszul exact
sequence.

Suppose that $0\to L\to E\to F\to0$ is an exact sequence of vector bundles on $Q$ of ranks $1,n$ and $n-1$ respectively. 
Then we get an exact sequence
\beq{lamm}
0\to L\otimes\Lambda^{m-1}F\to\Lambda^mE\to\Lambda^mF\to0.
\eeq
On tensoring with the determinant $\Lambda^nE^*\cong L^*\otimes\Lambda^{n-1}F^*$ and recalling that $n=2m+1$, this becomes the sequence
\beq{lamm2}
0\to\Lambda^{m+1}F^*\to\Lambda^{m+1}E^*\to L^*\otimes\Lambda^mF^*\to0.
\eeq
Then it is a standard fact that this is the same exact sequence as the one obtained by taking $\Lambda^{m+1}$ of the dual sequence $0\to F^*\to E^*\to L^*\to0$.

We apply this to $L,E$ and $F$ being $\Omega^{n-1}_Q(n-3),
\,\Omega^{n-1}_Y(\log Q)\restr Q(n-3)$ and $\Omega^{n-2}_Q(n-3)$ respectively.
The two quotient maps in the sequences \eqref{lamm} and \eqref{lamm2}
become (on composition with the restriction map to $Q$) the two residue maps in the Lemma.
\end{proof}

Combined with the cup product on cohomology this gives a commutative diagram $$
\spreaddiagramrows{-.5pc}
\spreaddiagramcolumns{1pc}
\xymatrix{
\bigotimes^m H^1\big(\Omega^{n-1}_Y(\log Q)(-(n-3)Q)\big) \ar[d]
\ar[r]^(.57){\mathrm{Res^{\otimes m}}}
& \bigotimes^m H^1(\Omega^{n-2}_Q(n-3))\ar[d] \\
H^m\big(\Omega^{m+1}_Y(\log Q)\big) \ar[r]\ar[r]^(.6){\mathrm{Res}} & H^m(\Omega^{m}_Q).}
$$
Since both wedging and cup product are skew-commutative, the vertical multiplication maps are now commutative. Composing with the $m$-th tensor product map $V\to V^{\otimes m}$ in each case we get a commutative diagram whose vertical maps are nonlinear:
\beq{resdiag}
\spreaddiagramrows{-.5pc}
\xymatrix{
H^1\big(\Omega^{n-1}_Y(\log Q)(-(n-3)Q)\big) \ar[d]^{\mu\_Y}\ar[r]^(.6){\mathrm{Res}}
& H^1(\Omega^{n-2}_Q(n-3))\ar[d]^{\mu\_Q}\\
H^m\big(\Omega^{m+1}_Y(\log Q)\big) \ar[r]\ar[r]^(.6){\mathrm{Res}} & H^m(\Omega^{m}_Q).}
\eeq
In particular we get the nonlinear map
\beq{Yuk}
\Ext^1(\Omega_X,\O_X)\overset{\mu\_Y}{\To} H^{m}\big(\Omega^{m+1}_Y(\log
Q)\big)\Rt{\mathrm{Res}}H^{m,m}(Q).
\eeq

\begin{rmk}
More generally one can define (commutative) rings
\beq{Ry}
R_Y:=\bigoplus _{k=0}^{n}H^k(\Lambda^kT_Y(2k\,Q))\isom \bigoplus_{k=0}^n
H^k(\Omega_Y^{n-k}((2k+2-n)Q))
\eeq
and
\beq{Rq}
R_Q:=\bigoplus _{k=0}^{n-1}H^k(\Lambda^kT_Q(-2k))\isom \bigoplus_{k=0}^{n-1}
H^k(\Omega_Q^{n-1-k}(n-1-2k))
\eeq
with product induced by the wedge product of polyvector fields and cup product
on cohomology. By some easily checked cohomology vanishing on $Q$ one sees
that in the definition \eqref{Rq} of $R_Q$ one can replace $\Lambda^*T_Q$
by $\Lambda^*T_Y\restr Q$ in every degree except $k=n-1$. 
Since the wedge and cup products commute with
restriction to $Q$, we can define a ring homomorphism $R_Y\to R_Q$. (To deal
with
the troublesome degree $n-1$ classes we can pick a complement in $H^{n-1}
(\Omega^1_Y(nQ))$ to the kernel of the composition
\[ H^{n-1}(\Omega^1_Y(nQ))\to H^{n-1}(\Omega^1_Y(nQ)\restr Q)\to H^{n-1}(\Omega^1_Q(-n))\]
and map this to zero in $R_Q$. The degree $n$ part of $R_Y$ is already zero.)
\end{rmk}

The ring structure on $R_Q$ can be determined explicitly and turns out to
be closely related  to the cohomology ring of the quadric.

Recall that a quadric $Q$ of dimension $n-1=2m$ carries two natural classes
of $m$-planes  which we call $A\cong\PP^m$ and $B\cong\PP^m$. For instance,
describing
$Q\subseteq\PP^{2m+1}$ with homogeneous coordinates $x_0,\ldots,x_m,y_0,\ldots
y_m$ as the zero locus of $\sum_{i=0}^mx_iy_i$ then we can take 
\beq{AB} 
A=\{x_0=0=x_1=\ldots=x_m\} \quad\text{and}\quad B=\{y_0=0=x_1=\ldots=x_m\}. 
\eeq 
Together their cohomology classes  (which we also call
$A$ and $B$) generate $H^m(\Omega^m_Q)$, which is the middle degree part
of the ring $R_Q$ \eqref{Rq}.

\begin{prop}\label{R_Q} Let $Q$ be a quadric of dimension $2m=n-1$. There is a generator $\eta$ of $H^1(T_Q(-2))$
such that
$$
R_Q\,=\,\langle\eta,\eta^2,\ldots,\eta^{m-1},A,B,\eta^{m+1},
\ldots,\eta^{n-1}\rangle
$$
with $\eta^m=A-B$ and $\eta(A+B)=0$.
\end{prop}

\begin{proof}
The normal bundle sequence, twisted by $\O_Q(-2)$, 
\[ 0\to T_Q(-2)\to T_{\PP^n}(-2)\restr Q\to \O_Q\to 0\]
induces, by Lemma \ref{restrcohom}, an isomorphism $H^0(\O_Q)\isom H^1(T_Q(-2))$.
So we can identify $\eta$ with (a nonzero multiple of) the extension
class of this sequence.

Therefore in the wedge powers of the conormal bundle sequence, twisted by
$\O_Q(n-1-2k)$, 
\[0\to \Omega_Q^{n-2-k}(n-3-2k)\to\Omega_{\PP^n}^{n-1-k}\restr Q(n-1-2k)\to\Omega_Q^{n-1-k}(n-1-2k)\to0,\]
the boundary map 
\[H^k(\Omega_Q^{n-1-k}(n-1-2k))\overset{\eta\lrcorner}{\To} H^{k+1}(\Omega_Q^{n-2-k}(n-3-2k))\]
is given by  contraction with $\eta$. This corresponds to multiplication by $\eta$ in $R_Q$ under the isomorphism $H^1(\Omega^{n-2}_Q(n-3))\isom H^1(T_Q(-2))$.

If  $n-1=2m$ is even then the kernel and cokernel of this map always vanish
by Lemma \ref{restrcohom} unless $k=m-1$ or $k=m$. In these degrees we get exact
sequences
$$
\spreaddiagramcolumns{-.5pc}
\spreaddiagramrows{-1pc}
\xymatrix{
0\ar[r]& H^{m-1}(\Omega^{m+1}_Q(2))\ar[r]^<(.2){\eta\lrcorner}& H^m(\Omega^m_Q)\ar[r]\ar[dr]_\gamma&
H^m(\Omega^{m+1}_{\PP^n}\restr{Q}) \ar[r]\ar@{=}[d]& 0\\
&&& H^{m+1}(\Omega^{m+1}_{\PP^n}),}
$$
\vskip -3mm\noindent and\vspace{-1mm}
$$
\spreaddiagramcolumns{-.5pc}
\spreaddiagramrows{-1pc}
\xymatrix{
& H^m(\Omega^m_{\PP^n}) \ar@{=}[d]\ar[dr]^{\alpha} \\
0 \ar[r]& H^m(\Omega^m_{\PP^n}\restr Q) \ar[r]& H^m(\Omega^m_Q)
\ar[r]^(.35){\eta\lrcorner}& H^{m+1}(\Omega^{m-1}_Q(-2)) \ar[r]& 0.}
$$
Since $\gamma$ is the Gysin homomorphism, with kernel generated
by $A-B$, we can arrange (by rescaling) that $\eta^m=A-B$. The image of $\alpha$
is the restriction of the $m$th power of the Fubini-Study class to $Q$ and
hence equal to $A+B$. Therefore $\eta(A+B)=0$.
\end{proof}

\begin{rmk}
We recognise this description of $R_Q$ as very much like $H^*(Q,\C)$, which
has a generator $\omega$ in degree $1$ and can be described as
\[
R_Q\,=\,\langle\omega,\omega^2,\ldots,\omega^{m-1},A,B,\omega^{m+1},
\ldots,\omega^{n-1}\rangle,
\]
with $\omega^m=A+B$ and $\omega(A-B)=0$ (notice the differences in sign from
Proposition \ref{R_Q}).

This is no coincidence: 
the quadric is  defined by a section of $\O_{\PP^n}(2)$;
differentiating twice along its zero set $Q$ gives a symmetric map
$T_Q\tensor T_Q\to \O_Q(2)$ which is nondegenerate. The resulting isomorphism
$T_Q(-2)\cong\Omega_Q$ induces an identification $R_Q\to\bigoplus_iH^i(\Omega^i_Q)=
H^*(Q, \C)$ which takes $\eta$ to $\omega$.

However in the middle degree $m$ this isomorphism  differs from the
one defined by the pairing $\Omega^m_Q\tensor\Omega^m_Q\to K_Q$ which was
employed in \eqref{Rq}; this explains
how it can interchange the classes $A-B$ and $A+B$. 
\end{rmk}

In the following we tacitly identify the local deformation space at $p_i$ with $\C$ via $\eta$. Recall that $2m+1=n\ge5$.

\begin{thm} \label{biggy} A class $e\in\Ext^1(\Omega_X,\O_X)$ that maps to
$(\epsilon_i)\in\bigoplus_iH^0(\O_{p_i})$ under \eqref{loctoglob}
maps to $\Big(\epsilon_i^{m}(A_i-B_i)\Big)\,\in\,\bigoplus_iH^{n-1}(Q_i)$ under \eqref{Yuk}.

Any such class gives rise to a relation
\beq{reln}
\sum_i\delta_i(A_i-B_i)=0\ \in H_{n-1}(Y),
\eeq
where $\delta_i=\epsilon_i^{m}$. In particular, $e$ is a (first order)  smoothing if and only if all $\delta_i$ are nonzero.
\end{thm}

\begin{proof}
The first statement follows immediately from the commutativity of \eqref{resdiag}
and the calculation $\eta^m=A-B$ of Proposition \ref{R_Q}.
We obtain the second statement by Poincar\'e duality from the cohomology
exact sequence
\[
\ldots\to H^m\big(\Omega^{m+1}_Y(\log Q)\big)
\Rt{\mathrm{Res\,}}H^{m,m}(Q)\to
H^{m+1,m+1}(Y)\to\ldots
\]
induced from
\[
0\to\Omega^{m+1}_Y\to\Omega^{m+1}_Y(\log Q)
\Rt{\mathrm{Res\,}}\Omega^m_Q\to0.
\]
The deformation smooths the double point $p_i$ if and only if its value
$\epsilon_i$ at $p_i$ under the map \eqref{loctoglob} is nonzero, if and
only if $\delta_i\neq 0$.
\end{proof}

To show that the Theorem is not vacuous we exhibit an example where not all $A_i-B_i$ classes are zero in $H_{n-1}(Y)$. In dimension $n$ we let $X$ denote Schoen's nodal Calabi-Yau hypersurface
$$
\{x_0^{n+2}+\ldots x_{n+1}^{n+2}-(n+2)x_0\ldots x_{n+1}=0\}\subset\PP^{n+1}.
$$
This is smooth except for $(n+2)^n$ ODPs at the points $[\zeta^{a_0}:\ldots,\zeta^{a_n}:1]$, where $\zeta=e^{2\pi i/(n+2)}$ and $\sum a_i\equiv 0$ mod $n+2$.

Then by \cite[Proposition 3.4]{Sch}, the $A_i-B_i$ classes span a subspace of dimension $(n+1)!$ in $H_{n-1}(Y)$. Since $X$ can certainly be smoothed to a smooth degree $n+2$ hypersurface of $\PP^{n+1}$, Theorem \ref{biggy} provides some of the $(n+2)^n-(n+1)!$ relations between the $A_i-B_i$ classes. In fact it linearly generates them \emph{all}, as we shall see in the next Section.


\section{Nodal hypersurfaces} \label{hyper} 
In this section we discuss a possible converse to Theorem \ref{biggy}.
Recall that we have the linear map
\beq{I}
\Ext^1(\Omega_X,\O_X)\to \bigoplus_iH^0(\O_{p_i})
\eeq
of \eqref{loctoglob}. Call its image $I\subseteq \bigoplus_iH^0(\O_{p_i})$. This is
the space of local smoothings of the ODPs that can be realised by a global
smoothing.

We also have the linear map taking the $A_i-B_i$ cycles
to their (co)homology classes in $Y$,
\beq{K}
 \bigoplus_iH^0(\O_{p_i})\cong\bigoplus_i\C.(A_i-B_i)\subseteq\bigoplus_i H^{m,m}(Q_i)\to
H^{m+1,m+1}(Y).
\eeq
Call its kernel $K\subseteq  \bigoplus_iH^0(\O_{p_i})$. This is the space of homology
relations amongst the $A_i-B_i$ cycles in $Y$.

Think of $\oplus_i\O_{p_i}$ as a semisimple algebra (by multiplication of
functions independently at each point $p_i$). Then we have the $m$th-power
map
\beq{mth}
(\epsilon_i)_{i=1}^N\mapsto(\epsilon_i^m)_{i=1}^N
\eeq
from $\oplus_i\O_{p_i}$
to itself. Here $N$ is the number of ODPs $p_i$. By Theorem \ref{biggy} this maps $I$ to $K$ (and
its composition with \eqref{I} gives $\mu\_Y$ of \eqref{Yuk}). Friedman's
theorem says that in dimension $n=3$ this map $I\to K$ is a (linear) isomorphism.

Since \eqref{mth} is nonlinear for $n\ne3$ it is not sensible
to ask for it to be an isomorphism. Simple calculations with nodal
hypersurfaces show that in general $\dim K>\dim I$, so we cannot expect it
to be onto. It makes more sense to ask for the following converse.

\begin{quest}
 Does the image of the map \eqref{mth} restricted
to $I$ linearly span $K$\,?
\end{quest}

Using the algebra structure on $\oplus_i\O_{p_i}$
we can talk about polynomials in the elements of $I\subseteq \bigoplus_iH^0(\O_{p_i})$.
Question 1 involves only polynomials of the form $\sum_ja_jx_j^m$ on elements
$x_j=(\epsilon_{j,i})_{i=1}^N$ of $I$. But by some elementary linear algebra
(for any vector space $V$, the symmetric power $S^mV$ is generated by elements
of the form $x^{\otimes m}$ for $x\in V$), Question 1 is equivalent to the
following \emph{a priori} weaker question. 

\begin{quest}
 Do degree $m$ polynomials on 
$I\subseteq \bigoplus_iH^0(\O_{p_i})$ generate $K$\,?  
\end{quest}

Another way of saying this is to consider the linear map
\[
S^mI\to \bigoplus_iH^0(\O_{p_i})
\]
induced from $I\into \bigoplus_iH^0(\O_{p_i})$ by the algebra structure on $\oplus_i\O_{p_i}$.
Then this maps into $K$, and we are asking whether or not it is \emph{onto}
$K$. 
This can also be phrased in terms of the ring homomorphism $R_Y \to R_Q$ introduced in \ref{Rq} and \ref{Ry}: we ask if $K$ is contained in the image of the subring of $R_Y$ generated by $H^1(T_Y(2Q))\isom \Ext^1(\Omega_X, \O_X)$.

In particular if the answer to Question 2 were true then, by the formulation
in Question 1, a homology relation
$$
\sum_i\delta_i(A_i-B_i)=0\in H_m(Y)
$$
with each $\delta_i\ne0$ would imply the existence of a deformation of $X$
which smooths \emph{every} ODP $p_i$. \medskip

We study this problem for nodal Calabi-Yau hypersurfaces
of projective space using Schoen's extension of the Griffiths-Dwork method
\cite{Sch}.

Fix an anticanonical (i.e. degree $n+2$) divisor $X\subset\PP^{n+1}$.
Suppose $X$ has only ODPs, and let $Z=\bigcup_i\{p_i\}$ be the nodal
set. As before let $Y$ denote the blow up of $X$ in $Z$ with exceptional
divisor $Q=\bigcup_iQ_i$.

Applying $\Hom(\ \cdot\ ,\O_X)$ to $0\to\O_X(-X)\to
\Omega_{\PP^{n+1}}\restr X\to\Omega_X\to0$ gives the exact sequence
$$
0\to H^0(T_{\PP^{n+1}}\restr X)\to H^0(\O_X(X))\to\Ext^1(\Omega_X,\O_X)\to0.
$$
The middle group is $H^0(\O_{\PP^{n+1}}(n+2))/\langle f\rangle$, where $f\in
H^0(\O_{\PP^{n+1}}(n+2))$ is the section defining $X$. Adding elements of
this group to $f$ gives a surjection to the first order deformations $\Ext^1(\Omega_X,\O_X)$
of $X$. The kernel $H^0(T_{\PP^{n+1}}\restr X)$ describes the infinitesimal
action of automorphisms
of $\PP^n$. The first arrow differentiates $f$ down a vector field
in $H^0(T_{\PP^{n+1}}\restr X)$, but $f$ vanishes to second order on $Z$,
so its image in fact lies in $H^0(\O_X(X)\otimes\I_Z)\subset H^0(\O_X(X))$.
This is easily seen to be precisely the subspace of deformations of the divisor
$X$ which preserve the ODPs $p_i$ to first order (i.e. they may move them,
but not smooth them). Therefore we get a commutative diagram
\beq{hyperdefs} \vspace{-7mm} \eeq
{\small $$
\spreaddiagramcolumns{-1pc}
\spreaddiagramrows{-1pc}
\xymatrix{
& 0 \ar[d]& 0 \ar[d] \\
& H^0(T_{\PP^{n+1}}\restr X) \ar@{=}[r]\ar[d]& H^0(T_{\PP^{n+1}}\restr X)
\ar[d] \\
0 \ar[r] & H^0(\O_X(X)\otimes\I_Z) \ar[d]\ar[r]& H^0(\O_X(X)) \ar[r]\ar[d]&
\bigoplus_iH^0(\O_{p_i}) \ar[r]\ar@{=}[d]& H^1(\O_X(X)\otimes\I_Z) \ar[d] \\
0 \ar[r] & H^0(T_X) \ar[r]\ar[d]& \Ext^1(\Omega_X, \O_X) \ar[r]\ar[d]&
\bigoplus_iH^0(\O_{p_i}) \ar[r]& H^2(T_X) \\
& 0 & 0
}$$}

Schoen's generalisation \cite{Sch} of the Griffiths-Dwork method
of generating primitive cohomology of hypersurfaces via residues gives, with
a little work, the following commutative diagram.
\beq{schoen}
\spreaddiagramcolumns{-1pc}
\spreaddiagramrows{-1pc}
\xymatrix{
0 \ar[r] & H^0(\I_Z(mX)) \ar[d]\ar[r]& H^0(\O_X(mX))
\ar[r]\ar[d]& \bigoplus_iH^0(\O_{p_i}) \ar[r]\ar@{=}[d]& H^1(\I_Z(mX))
\ar@{=}[d] \\
0 \ar[r] & H^m(\Omega^{m+1}_Y) \ar[r]\ar[d]& H^m(\Omega^{m+1}_Y(\log
Q)) \ar[r]\ar[d]& \langle A_i-B_i\rangle_Q \ar[r]& \langle A_i-B_i\rangle_Y
\\ & 0 & 0}
\eeq
Here the cohomology groups
in the upper row are calculated on $X$.
(To make contact with \cite{Sch} one should, for instance, write $H^0(\O_X(mX))$
as $H^0(K_{\PP^{n+1}}((m+1)X))\big/H^0(K_{\PP^{n+1}}(mX))$
and take residues on $X$ to get the vertical maps.)
The penultimate
term on the bottom row is the span of
the $A_i-B_i$ classes in $H^m(\Omega^m_Q)=
\bigoplus_i\langle A_i,B_i\rangle$. The
final term is their span after push forward to $H^{m+1}(\Omega^{m+1}_Y)$.

In this case $I$, as defined in \eqref{I}, is the image
of $H^0(\O_X(X))$ in the lower $\bigoplus_iH^0(\O_{p_i})$ in the diagram \eqref{hyperdefs}.
And $K$ \eqref{K} is, by the exactness of the diagram \eqref{schoen}, the
image of
$H^0(\O_X(mX))$ in $\langle A_i-B_i\rangle_Q$. The map from $I$ to $K$ factors
through the following commutative diagram
$$
\xymatrix{H^0(\O_X(X)) \ar[r]\ar[d] & H^0(\O_X(mX)) \ar[d] \\
\bigoplus_iH^0(\O_{p_i}) \ar[r] & \bigoplus_i\C(A_i-B_i).}
$$
The vertical maps are those coming from the diagrams (\ref{hyperdefs}, \ref{schoen}),
while both horizontal maps take the $m$th power. However, for hypersurfaces
$X$ of projective space, the map
$$
S^m H^0(\O_X(X))\to H^0(\O_X(mX))
$$
is onto. It follows that Question 2 can be answered positively for such $X$,
and therefore the converse (Question 1) is true in this case.

\medskip
That the converse is true for hypersurfaces is the nodal analogue of the
fact that for smooth Calabi-Yau hypersurfaces of projective
space, $H^1(T_Y)$ generates $H^m(\Lambda^mT_X)\cong H^{m+1,m}(X)$. (This
is most easily proved by the Griffiths-Dwork method.) For general smooth
Calabi-Yau manifolds it is dual to the condition that on any mirror
manifold $Y$, $H^2(Y,\C)$ should generate $H^{2*}(Y,\C)$ under the quantum
cohomology product. It also seems that it is probably
not true in general (as Lev Borisov explained to us) even for a Calabi-Yau
hypersurface in a toric variety $\PP_\Delta$, for instance when some lattice
points in the multiples of the reflexive
polytope $\Delta$ are not sums of lattice points in $\Delta$. Therefore
we think
it unlikely that the answer to Questions 1 or 2 is positive in general.

However, it is still possible that a weaker condition might hold; namely
one could ask whether for deformations of any given \emph{smooth} Calabi-Yau
manifold $X$,

\begin{quest}
\text{Does the image of the period map span $H^n(X,\C)$ ?}
\end{quest}

\noindent It seems not to be known whether one should expect this in general or not.

If this were true then a finite number of the derivatives
of the
period point at any given $(X,\Omega)$ would generate $H^n(X,\C)$.
In this case we would expect that also for \emph{nodal} $X$, some high order
deformation
of the complex structure would produce a holomorphic $n$-form with a nonzero
coefficient of the dual of the class $\Delta$. Its pullback to a
topological model of the smoothing would then contain PD$(\widetilde\Delta)$
and so would have nonzero integral against the vanishing cycles. It would
thus correspond to a genuine smoothing of $X$.

So if the answer to Question 3 is positive for \emph{all smooth Calabi-Yau manifolds} then we
still expect a weak converse for nodal $X$. That is,
a homology relation as in Theorem \ref{biggy}, with all $\delta_i\ne0$, would
imply the existence of a smoothing.

\appendix
\section{Cohomology calculations} \label{app}

In this appendix we compute some cohomology groups of twists of sheaves
of holomorphic forms on a quadric hypersurface $Q\subset \PP^n$. For simplicity
we will always assume that $n\geq 5$. 

\begin{prop}\label{cohomQ} The cohomology groups $H^{\le2}(\Omega_Q^{n-1}(j))$
and $H^{\le2}(\Omega_Q^{n-2}(j))$ vanish for all $j\le n-2$ except
$H^1(\Omega_Q^{n-2}(n-3))\cong\C$.
\end{prop}

The proof consists mainly in book-keeping in the long exact cohomology sequences
of both
\begin{equation}\label{restriction}
0 \to  \Omega_{\PP^n}^k(j-2) \to \Omega_{\PP^n}^k(j) \to  \Omega_{\PP^n}^{k}(j)
\restr{Q} \to 0
\end{equation}
and the (twisted) wedge powers of the conormal bundle sequence,
\begin{equation}\label{conormal}
0\to\Omega_Q^{k-1}(j-2)\to \Omega_{\PP^n}^{k}(j)\restr{Q}\to\Omega^k_Q(j)\to0.
\end{equation}

First of all we need to recall Bott's formula for the dimension of the various cohomology groups on projective space:
\begin{thm}[\cite{Bott}] Let $0\leq p\leq n$ and $m\in\Z$. Then   $h^q(\PP^n, \Omega_{\PP^n}^p(m))$ vanishes for all $q$ apart from  $h^p(\Omega_{\PP^n}^p)=1$ and 
\begin{gather*}
h^0(\PP^n, \Omega_{\PP^n}^p(m))=
\binomial{m-1}{p}\binomial{m+n-p}{m}=h^n(\PP^n, \Omega_{\PP^n}^{n-p}(-m))
\end{gather*}
for  $p<m$.
%
\end{thm}

\begin{lem}\label{restrcohom}
 The cohomology of $\Omega^p_{\PP^n}(j)\restr Q$ is
\begin{itemize}
 \item[$j<0$:] $h^q(\Omega^p_{\PP^n}(j)\restr Q)= 0$ unless $q=n-1$ and $j<p+2-n$.
\item[$j=0$:]
$h^q(\Omega^p_{\PP^n}(0)\restr Q)=
\begin{cases}
1 & 0 \leq p=q\leq n-1\\
0 &  \text{otherwise}
\end{cases}$.
 \item[$j=1$:] $h^q(\Omega^p_{\PP^n}(1)\restr Q)=
\begin{cases}
n+1 & q=n-1, p=n \text{ or } q=p=1\\
0 &  \text{otherwise}
\end{cases}$.
\item[$j=2$:]
$h^q(\Omega^p_{\PP^n}(2)\restr Q)=0 $ unless $p=q=0$ or $0\leq q=p-1\leq n-1$.
 \item[$j>2$:] $h^q(\Omega^p_{\PP^n}(j)\restr Q)=0$ unless $ q=0$ and $j>p$.
\end{itemize}

\end{lem}
\begin{proof} All statements follow pretty directly from the long exact sequence associated to the restriction sequence \eqref{restriction}, so we will give the details only for the last assertion.

If $j>2$ then neither $\Omega^p_{\PP^n}(j)$ nor $\Omega^p_{\PP^n}(j-2)$ can have higher cohomology and the only nonzero part of the long exact sequence is
\[ 0\to H^0(\PP^n,\Omega^p_{\PP^n}(j-2))\to H^0(\PP^n,\Omega^p_{\PP^n}(j)) \to H^0(\PP^n,\Omega^p_{\PP^n}(j)\restr Q)\to 0,\]
which immediately implies the assertion.\end{proof}

\emph{Proof of Proposition (\ref{cohomQ})}.
The $j=0,p<0$ part
of the preceding Lemma \ref{restrcohom} gives the required cohomology
vanishing
for the twists of $\Omega^{n-1}_Q$, since $\O_{\PP^n}\restr{Q}=\O_Q\isom
\Omega^{n-1}_Q(n-1)$.
%

For the twists of $\Omega^{n-2}_Q$ we use the sequences \eqref{restriction}
and \eqref{conormal} with $k=n-1,n-2$. Using the assumption that $n\geq 5$
we see that 
\[H^{1}(\Omega^{n-1}_{\PP^n}(j)\restr Q)=H^{1}(\Omega^{n-1}_Q(j))=H^{2}(\Omega^{n-1}_{\PP^n}
(j)\restr Q)=0.\]
Thus also $H^2(\Omega_Q^{n-2}(j-2))=0$ and the relevant terms of the long
exact sequence are
\begin{multline*}
0\to H^0(\Omega_Q^{n-2}(j-2))\to H^0(\Omega_{\PP^n}^{n-1}(j)\restr Q) \to
H^0(\Omega_{Q}^{n-1}(j)) \\ \to H^1(\Omega_Q^{n-2}(j-2))\to 0.
\end{multline*}
The central two cohomology groups vanish for $j\leq n-2$ so we only have
to analyse the cases $j=n-1, n$. If $j=n-1$ then still $H^0(\Omega_{\PP^n}^{n-1}(j)\restr Q)=0$ so $H^0(\Omega_Q^{n-2}(n-3))=0$ and  
\[H^1(\Omega_Q^{n-2}(n-3))\isom H^0(\Omega_{Q}^{n-1}(n-1))\isom H^0(\O_Q)\isom
\C.\]
If $j=n$ then $H^0(\Omega_{\PP^n}^{n-1}(n)\restr Q) \to H^0(\Omega_{Q}^{n-1}(n))$
is the map $H^0(T_{\PP^n}(-1))\to H^0(\O_Q(1))$ which differentiates the
quadratic form $Q$ (thought of as the section of $\O(2)$ defining the quadric
$Q$) down vector fields in $\PP^n$. This is nothing but the isomorphism $\C^{n+1}
\to(\C^{n+1})^*$ induced by the nondegenerate quadratic form $Q$ on $\C^{n+1}$.
So $H^0(\Omega_Q^{n-2}(n-2))=H^1(\Omega_Q^{n-2}(n-2))=0$.\qed

\end{document}